\documentclass[a4paper]{amsart}

\usepackage{amsmath,amsthm,amssymb,color,enumerate,multicol,datetime,todonotes}

\usepackage{url}

\theoremstyle{plain}
\newtheorem{theorem}{Theorem}[section]
\newtheorem{lemma}[theorem]{Lemma}
\newtheorem{proposition}[theorem]{Proposition}
\newtheorem{corollary}[theorem]{Corollary}

\theoremstyle{definition}
\newtheorem{definition}[theorem]{Definition}
\newtheorem{example}[theorem]{Example}

\theoremstyle{remark}

\newcommand{\med}{\mathrm{med}}
\newcommand{\bfa}{\mathbf{a}}
\newcommand{\bfb}{\mathbf{b}}
\newcommand{\bfc}{\mathbf{c}}
\newcommand{\bfx}{\mathbf{x}}

\newcommand{\Cl}{\mathcal{O}}
\newcommand{\Clf}[2]{\smash{\mathcal{O}^{(#1)}_{#2}}}
\newcommand{\struc}[1]{\langle #1 \rangle}

\date{\today,\currenttime}
\title{Pivotal decomposition schemes inducing clones of operations}

\author{Miguel Couceiro}
\address{LORIA, (CNRS - Inria Nancy Grand Est - Universit\'e de Lorraine),
BP239~-~54506 Vandoeuvre les Nancy, France}
\email{miguel.couceiro@\{loria,inria\}.fr}
\author{Bruno Teheux}
\address{Mathematics Research Unit, FSTC,  University of Luxembourg, 6, rue Coudenhove-Kalergi, L-1359 Luxembourg, Luxembourg}
\email{bruno.teheux@uni.lu}

\begin{document}
%\normalem

\begin{abstract}
We study pivotal decomposition schemes and investigate classes of pivotally decomposable operations. 
%We show that among the latter, many constitute clones. 
We  provide sufficient conditions on pivotal operations that 
guarantee that the corresponding classes of pivotally decomposable operations are clones, and show that under certain assumptions
these conditions are also necessary. In the latter case, the pivotal operation together with the constant operations  generate the corresponding clone. 
\end{abstract}

\maketitle

\section{Introduction and Motivation}

Several classes of operations have the remarkable feature that each member $f\colon A^n\to A$ is decomposable into simpler operations that are then combined by a single operation,
in order to retrieve the values of the original operation $f$. A noteworthy example is the class of Boolean functions $f\colon \{0,1\}^n\to \{0,1\}$   that can be decomposed into 
expressions of the form
\begin{equation}
 f(\mathbf{x})=x_kf(\mathbf{x}^1_k)+(1-x_k)f(\mathbf{x}^0_k),
 \label{eq:Shancomp}
\end{equation}
for $\mathbf{x}=(x_1,\dots,x_n)\in \{0,1\}^n$ and $k\in[n]$, and  
where $\mathbf{x}^c_k$ denotes the $n$-tuple obtained from $\mathbf{x}$ by substituting its $k$-th component by $c\in \{0,1\}$. 
Such decomposition scheme is referred to as \emph{Shannon decomposition} (or \emph{Shannon
expansion}) \cite{Shannon}, or \emph{pivotal decomposition} \cite{Barlow}.
Boolean functions are similarly  decomposable into expressions in the language of Boolean lattices
\begin{equation}
 f(\mathbf{x})=(x_k\wedge f(\mathbf{x}^1_k))\vee (\overline{x}_k\wedge f(\mathbf{x}^0_k))
 \label{eq:latcomp}
\end{equation}
where $\overline{x}_k=1-{x}_k$.  

More recent examples include the class of polynomial operations over a distributive lattice (essentially, combinations of variables and constants using the lattice 
operations $\wedge$ and $\vee$) that were shown in~\cite{Marichal} to be decomposable into expressions of the form
\begin{equation}
f(\mathbf{x}) = 
\med(f(\mathbf{x}^0_k), x_k,
     f(\mathbf{x}^1_k)),
\label{eq:medcomp}
\end{equation}
where $\med$ is the ternary lattice polynomial given by
\begin{align*}
\med(x_1, x_2, x_3)
&= (x_1 \wedge x_2) \vee (x_1 \wedge x_3) \vee (x_2 \wedge x_3) \\
&= (x_1 \vee x_2) \wedge (x_1 \vee x_3) \wedge (x_2 \vee x_3). %\\
%&= x_1 x_2 + x_1 x_3 + x_2 x_3 .
\end{align*}
The latter decomposition scheme is referred to as \emph{median decomposition} in \cite{CM2009} and \cite{Marichal}. We refer the reader to \cite{Miyata,Tohma,CLMW} for applications of
the median decomposition formula to obtain median representations of Boolean functions.

Note that  decomposition schemes (\ref{eq:Shancomp}),  (\ref{eq:latcomp}) and (\ref{eq:medcomp}) share the same general form, namely,
$$f(\mathbf{x}) = \Pi (x_k, f(\mathbf{x}^1_k), f(\mathbf{x}^0_k)).$$
Indeed,
\begin{itemize}
 \item  in (\ref{eq:Shancomp}) we have $ \Pi (x, y, z)= x y + (1 - x) z$,
 \item in (\ref{eq:latcomp})  we have $ \Pi (x, y, z)= (x\wedge y) \vee (\overline{x}\wedge z)$, and
\item  in  (\ref{eq:medcomp}) we have $ \Pi (x, y, z)=\med(x,y,z)$.\\
 \end{itemize}

These facts were observed in \cite{Marichal2014} where such pivotal decomposition schemes were investigated.
These preliminary efforts were then further pursued under the observation  that certain classes of pivotal operations
fulfill certain closure requirements, notably, closure under functional composition. This led to the study \cite{CT2015} of 
those classes of pivotally decomposable operations that constitute clones. In particular, we presented conditions on pivotal operations
to ensure that the corresponding classes of pivotally decomposable operations constitute clones. However, several questions were stated without
being answered. In this paper we settle many of these questions and provide new insights in this line of research.  
% 
% In this paper we are interested in classes of pivotally decomposable operations that constitute clones of operations over a set $A$.
% Our motivation is rooted in \cite{CFL} where a study of normal form representations of Boolean functions was presented and based on compositions of Boolean clones.
% In particular, it was shown that normal form representations of Boolean functions that have the ternary operator $\med$ as the only logical connective, allow 
% shorter representations than the classical DNF, CNF and polynomial representations.
% 
% As we will see,  clones of pivotally decomposable operations provide nice normal form representations having a unique operation as logical connective, namely, 
% the corresponding pivotal operation.

The paper is organised as follows. In Section \ref{sec:Basic} we recall basic notions and terminology 
that will be used throughout the paper (Subsection \ref{sec:preliminaries}). 
We also introduce the concepts of pivotal operation and that of pivotally decomposable class (Subsection \ref{sec:Pivotal1}) and discuss normal form representations that arise from 
such pivotal decompositions (Subsection \ref{sec:Pivotal2}). Moreover, we investigate certain symmetry properties that are common to pivotal operations (Subsection \ref{sec:Pivotal3}).
In Section \ref{sec:Clones} we consider the problem of describing classes of pivotally decomposable operations that are clones.
A general solution to this problem still eludes us, but we provide several sufficient conditions on pivotal operations that ensure the latter 
(Subsection \ref{sub:suff}).
In fact,  we show that under certain assumptions many of these conditions are also necessary (Subsection \ref{sub:nec}). 
The question of determining sets of generators for 
clones of pivotally decomposable operations is also addressed and partially answered. 
Taking this framework further into the realm of clone theory, many natural questions emerge.
For instance, we construct an example of a pivotal operation $\Pi$ for which the class of $\Pi$-decomposable operations is a clone that does not contain $\Pi$
(Subsection \ref{sub:nec2});
such an example is shown not to exist in the case of Boolean functions. 
Further questions that remain open are then discussed in Section \ref{sec:final}.

\section{Basic notions and notation}
\label{sec:Basic}

In this section we recall basic terminology used throughout the paper. In particular, we introduce the concepts 
of pivotal operation and  of pivotally decomposable class, and we observe that, under certain conditions, pivotal decompositions
lead to normal form representations that use a unique non trivial  connective, namely, the pivotal operation.
In the last subsection we investigate symmetric properties of pivotal operations and present some characterizations.

\subsection{Preliminaries: Clones of operations}\label{sec:preliminaries}
For any positive integer $n$, we denote by $[n]$ the set $\{1, \ldots, n\}$. For a nonempty set $A$, a function $f\colon A^n \to A$ 
is called an \emph{$n$-ary operation} on $A$. We denote by $\Clf{n}{A}$ the set of $n$-ary operations on $A$ and 
by $\Cl_A=\bigcup_{n\geq 1}\Clf{n}{A}$ the set of \emph{operations} on $A$.  
For any $f \in \Clf{n}{A}$, $S\subseteq [n]$ and $\bfa\in A^n$ we define the \emph{$S$-section $f_S^{\bfa}$ of $f$} as the $|S|$-ary operation on $A$ defined by  
$f_S^{\bfa}(\bfx)=f(\bfa_S^{\bfx})$, where $\bfa_S^{\bfx}$ is the $n$-tuple whose $i$-th coordinate is $x_i$, if $i\in S$, and $a_i$, otherwise.  
%An operation $g\in \Cl_A$ is called a \emph{section of} $f\in \Clf{n}{A}$ if there are some $S\subseteq [n]$, $\bfa \in A^n$ such that $f=f_S^{\bfa}$.  
%Informally, a section of $f$ is a function which can be obtained from $f$ by replacing some of its arguments by constants.  
%If $S=\{k\}$ for some $k\in [n]$, we say that $g$ is a \emph{unary section of $f$} and we write $f_k^\bfa$ for $f_{\{k\}}^\bfa$. 
%We say that a section $f_S^\bfa$ of $f$ is  \emph{essential} if there is a $\bfb \in A^n$ such that \smash{$f_S^\bfb$} is non-constant.  
For $k \in [n]$, we say that the $k$-th argument of $f\in \Clf{n}{A}$ is \emph{essential} if there is a tuple $\bfb\in A^n$ such that $f_{k}^{\bfb}$ is non-constant.
Otherwise, we way that it is \emph{inessential}.
%An operation $f\in \Cl_A$ is said to be \emph{essentially $n$-ary} if it has $n$ essential arguments.  

A \emph{clone} on $A$ is a set $\mathcal{C} \subseteq \mathcal{O}_A$  of operations on $A$ that 
\begin{enumerate}
 \item contains all projections  on $A$, i.e., 
operations $p^n_i\colon A^n\to A$ given by 
$$
p^n_i(x_1,\dots,x_n)=x_i,\, \text{ for $i\in [n]$, and}
$$
\item is closed under taking functional compositions, i.e., if $f\in C\cap \Clf{n}{A}$ and $g_1,\dots, g_n\in C\cap \Clf{m}{A}$, then their \emph{composition}
$ f(g_1,\dots, g_n)\in \Clf{m}{A}$ that is defined by   
\begin{eqnarray*}
 f(g_1,\dots, g_n)(\mathbf{x})=f(g_1(\mathbf{x}),\dots, g_n(\mathbf{x}))\quad (\bfx \in A^m)
\end{eqnarray*}
also belongs to $C$.
\end{enumerate}
In the case when $A$ is finite, the set of all clones on $A$ forms an algebraic lattice, 
where the lattice operations are the following: meet is the intersection, join is the smallest clone that contains the union. 
The greatest element is the clone $\mathcal{O}_A$ of all operations on $A$; the least element is the clone $\mathcal{J}_A$ of all projections on $A$. 
For sets $A$ of cardinality at least $3$, this lattice is uncountable, and its structure remains a topic of current research; see, e.g., \cite{DW,Lau}.
In the case when $\lvert A \rvert = 2$, the lattice of clones on $A$ is countably infinite, and it was completely described by E.~Post \cite{Post}.
In particular, it follows that each Boolean clone can be generated by a finite set of Boolean functions. For instance, 
\begin{itemize}
 \item the clone $\mathcal{O}_{\{0,1\}}$ of all Boolean functions can be generated by $\{\neg,\wedge\}$ or, equivalently, by $\{0,\neg,\med\}$;
  \item the clone $M$ of all monotone Boolean functions, i.e., verifying 
  $\mathbf{x}\leq \mathbf{y}\implies f(\mathbf{x})\leq f(\mathbf{y})$, 
  can be generated by $\{0,1,\wedge,\vee\}$ or, equivalently, by $\{0,1,\med\}$;
   \item the clone $SM$ of all self-dual monotone Boolean functions, i.e., monotone operations verifying 
  $ f(\neg \mathbf{x})=\neg f(\mathbf{x})$, 
  is  generated by  $\{\med\}$.
\end{itemize}
 For further background see, e.g.,  \cite{DW,Lau}.
%The clones of Boolean functions and the lattice of clones on $\mathbb{B}$ are often called \emph{Post classes} and the \emph{Post lattice,} respectively. 

\subsection{Pivotal operations and pivotally decomposable classes}
\label{sec:Pivotal1}

In what follows, $A$ denotes an arbitrary fixed nonempty set, and $0$ and $1$ are two fixed elements of $A$.  
In the setting of operations, the notion of pivotal operation  $\Pi$ and that of $\Pi$-decomposable operation can be defined as follows.

\begin{definition}[Definition 2.1 in \cite{Marichal2014}] 
A \emph{pivotal operation} on $A$ is a ternary operation $\Pi$ on $A$ that satisfies the equation
\begin{equation}\label{eqn:pri}
\Pi(x,y,y)\ = \ y.
\end{equation}
If $\Pi$ is a pivotal operation, then $f\in \Clf{n}{A}$  is \emph{$\Pi$-decomposable} if 
\begin{equation}\label{eqn:dec}
f(\bfx)\ = \ \Pi(x_i, f(\bfx_i^1), f(\bfx_i^0)), \qquad \bfx \in A^n, i\in [n].
\end{equation}
Also, we denote by $\Lambda_\Pi$ the class of $\Pi$-decomposable operations on $A$.
\end{definition}

Note that condition \eqref{eqn:pri} ensures that $\Pi$-decomposability of an operation does not depend on its inessential arguments. 
Indeed, if the $i$th argument of $f$ is inessential, then  $f(\bfx)=f(\bfx_i^1)=f(\bfx_i^0)$ for every $\bfx \in A^n$. 
It follows from  (\ref{eqn:pri}) that $f(\bfx)\ = \ \Pi(x_i, f(\bfx_i^1), f(\bfx_i^0))$ for any $\bfx \in A^n$. In particular, we can state the following result.

\begin{lemma}\label{lem:sta01}
If $\Pi$ is a pivotal operation, then every constant operation on $A$ is $\Pi$-decomposable.
\end{lemma}

\subsection{Normal form representations induced by pivotal decompositions}
\label{sec:Pivotal2}
Note that if an operation $f$ is $\Pi$-decomposable, then  we arrive at a representation of $f$
by an expression built from the pivotal operation $\Pi$ and applied to variables and constants, by iterating its $\Pi$-decomposition expression \eqref{eqn:dec}. 
This fact motivates the following notion of 
$\Pi$-normal form.

\begin{definition}
Let $\Pi \in \Clf{3}{A}$. We define the classes of \emph{$k$-ary $\Pi$-normal forms $N_{\Pi}^k$} inductively on $k\geq 0$  by the following rules.
\begin{enumerate}
\item $N_{\Pi}^0=\Clf{0}{A}$.
\item For any $k\geq 0$, the class $N_{\Pi}^{k+1}$ is defined by 
$$N_{\Pi}^{k+1}=\{\Pi(x_{k+1}, g, g')\mid g,g' \in N_{\Pi}^k\}.$$ 
\end{enumerate}
We denote by $N_\Pi$ the class $\bigcup_{k\geq 0}N_\Pi^k$ of the \emph{$\Pi$-normal forms}.
\end{definition}

Observe that $N_\Pi^k\subseteq \Clf{k}{A}$ for every $k \geq 0$. By repeated applications of \eqref{eqn:dec}, we get the following result.

\begin{proposition}\label{prop:norm}
If $\Pi$ is a pivotal operation, then $\Lambda_\Pi\subseteq N_\Pi$.
\end{proposition}

\subsection{Symmetric pivotal operations}
\label{sec:Pivotal3}

As we will see later in the paper, a pivotal operation $\Pi$ is not necessarily  $\Pi$-decomposable. 
However, when it is, then it verifies certain symmetry properties. Consider the following equations:
\begin{gather}
\Pi(x,y,z)=\Pi(z,x,y)\label{eqn:sym01},\\
\Pi(x,y,z)=\Pi(z,y,x)\label{eqn:sym02},\\
\Pi(x,1,0)=x \label{eqn:ex01},
\end{gather}
Clearly, $\Pi$ is  symmetric if and only if it satisfies \eqref{eqn:sym01} and \eqref{eqn:sym02}. The following result states that if $\Pi\in \Lambda_\Pi$ and satisfies  \eqref{eqn:sym01} and \eqref{eqn:ex01}, then it is symmetric.
\begin{proposition}\label{prop:sym01}
If $\Pi$ is a $\Pi$-decomposable pivotal operation that satisfies \eqref{eqn:ex01} and \eqref{eqn:sym01}, then it satisfies \eqref{eqn:sym02}. In particular, $\Pi$ is a symmetric operation.
\end{proposition}
\begin{proof}
We obtain successively
\begin{eqnarray}
\Pi(x,y,z) & = & \Pi(x, \Pi(1,y,z), \Pi(0,y,z))\label{eqn:frt01}\\
 & = & \Pi(x, \Pi(y, \Pi(1,1,z),\Pi(1,0,z)), \Pi(z, \Pi(0,y,1), \Pi(0,y,0)))\label{eqn:frt02}\\
& = & \Pi(x, \Pi(y,1,z), \Pi(z,y,0))\label{eqn:frt03}\\
& = & \Pi(x, \Pi(z,y,1), \Pi(z,y,0))\label{eqn:frt04}\\
& = & \Pi(z,y,x).\label{eqn:frt05}
\end{eqnarray}
where \eqref{eqn:frt01}, \eqref{eqn:frt02} and \eqref{eqn:frt05} were obtained by $\Pi$-decomposability of $\Pi$, and where  \eqref{eqn:frt03} and \eqref{eqn:frt04} were obtained by \eqref{eqn:pri}, \eqref{eqn:ex01} and \eqref{eqn:sym02}. 
\end{proof}

Under the assumption of $\Pi$-decomposability of $\Pi$ and  \eqref{eqn:ex01}, symmetry of a pivotal operation $\Pi$  can be characterized in the following way.
\begin{theorem}\label{thm:symsy}
Le $\Pi$ be a $\Pi$-decomposable pivotal operation that satisfies \eqref{eqn:ex01} .  The following conditions are equivalent. 
\begin{enumerate}[$(i)$]
\item\label{it:uyt01} $\Pi$ is symmetric.
\item\label{it:uyt02} $\Pi$ satisfies the  equations
\[
\Pi(0,1,0)~=~\Pi(0,0,1) \quad \text{ and } \quad \Pi(1,1,0)~=~\Pi(1,0,1).
\]
\end{enumerate}
\end{theorem} 
\begin{proof}
It is clear that (\ref{it:uyt01}) implies (\ref{it:uyt02}). Let us prove that (\ref{it:uyt02}) implies (\ref{it:uyt01}). It suffices to prove that $\Pi$ satisfies  
\begin{align}
\Pi(x,y,z)&=~\Pi(x,z,y),\label{eqn:sym02bis01}\\
\Pi(x,y,z)&=~\Pi(y,x,z).\label{eqn:sym02bis02}
\end{align}
First, note that for every $x\in A$  we obtain successively 
\begin{align}
\Pi(x, 1, 0)& =~\Pi(x, \Pi(1,1,0), \Pi(0,1,0))~=~\Pi(x, \Pi(1,0,1), \Pi(0,0,1))\nonumber\\
& =~\Pi(x,0,1),\label{eq:ggg01}
\end{align}
where the first identity is obtained by \eqref{eqn:sym02}, the second by contition (\ref{it:uyt02}) and the last one by $\Pi$-decomposability of $\Pi$.
Then, for every $y \in A$ we have
\begin{align}
\Pi(x,y,0) &=~\Pi(y, \Pi(x,1,0), \Pi(x,0,0))~=~\Pi(y, \Pi(x,0,1), \Pi(x,0,0)) \label{eq:ggg0200}\\
& =~\Pi(x,0,y),\label{eq:ggg02}
\end{align}
where the first and last identities are obtained by  decomposability of $\Pi$, and the second by \eqref{eq:ggg01}. Using a similar argument, we obtain
\begin{equation}\label{eq:ggg03}
\Pi(x,y,1)~=~\Pi(x,1,y).
\end{equation}
Finally, we obtain for any $z\in A$
\begin{align}
\Pi(x,y,z) &=~\Pi(z, \Pi(x,y,1), \Pi(x,y,0))~=~\Pi(z, \Pi(x,1,y), \Pi(x,0,y))\label{eq:ggg04} \\
&=~\Pi(x,z,y),\nonumber
\end{align}
where the first and last identities are obtained by  decomposability \eqref{eqn:dec} of $\Pi$, and the second by \eqref{eq:ggg02} and \eqref{eq:ggg03}. This proves that \eqref{eqn:sym02bis01} holds. Now, using \eqref{eqn:ex01} and \eqref{eqn:pri} in the first identity in \eqref{eq:ggg0200} we obtain that 
\begin{equation}\label{eq:ggg05}
\Pi(x,y,0)~=~\Pi(y,x,0).
\end{equation}
Similarly, we have
\begin{align}
\Pi(x,y,1)&=~\Pi(y,\Pi(x,1,1), \Pi(x,0,1))~=~\Pi(y,1, x),\nonumber	\\
&=~\Pi(y,x,1),\label{eq:ggg06}
\end{align}
where the first identity is obtained by decomposability of $\Pi$,  the second by \eqref{eqn:pri}, \eqref{eq:ggg01} and \eqref{eqn:ex01}, and the last one by \eqref{eq:ggg03}. Using \eqref{eq:ggg05} and \eqref{eq:ggg06} in the first identity of \eqref{eq:ggg04}, we obtain \eqref{eqn:sym02bis02} by $\Pi$-decomposability of $\Pi$.
\end{proof}

\section{Clones of pivotally decomposable operations}
\label{sec:Clones}

%We aim to give %(necessary) and sufficient  conditions on a pivotal operation $\Pi$ for $\Lambda_\Pi$ to be a clone.  
In Subsection \ref{sub:suff}, we provide sufficient conditions on a pivotal operation $\Pi$ for $\Lambda_\Pi$ to be a clone, 
and in Subsection \ref{sub:nec}, we prove that these conditions are also necessary under the assumption that $\Pi$ belongs to $\Lambda_\Pi$, and satisfies two additional equations \eqref{eqn:fine01} and \eqref{eqn:fine02} that involves only $\Pi$, and the elements $0$ and $1$.
Certain natural questions are also discussed  and answered negatively by counter-examples that are constructed in Subsection \ref{sub:nec2}. 

\subsection{Sufficient conditions for $\Lambda_\Pi$ to be a clone}\label{sub:suff}

Let us consider the following equations:
\begin{gather}
%\Pi(x,1,0)=x \label{eqn:ex01},\\
\Pi(\Pi(x,y,z),t,u)=\Pi(x,\Pi(y,t,u), \Pi(z,t,u))\label{eqn:ex04}.
%\Pi(\Pi(x,y,z),t,u)=\Pi(x,\Pi(\Pi(0,y,z),t,u),\Pi(\Pi(1,y,z),t,u))\label{eqn:ex05},
\end{gather}

The relevance of property (\ref{eqn:ex04}) is made apparent by the following lemma.

\begin{lemma}\label{lem:sta02}
Let $\Pi$ be a pivotal operation that satisfies \eqref{eqn:ex04}. If $f\colon A^n \to A$ and $g_1, \ldots, g_n\colon A^m\to A$ are $\Pi$-decomposable, 
then so is $f(g_1, \ldots, g_n)$.
\end{lemma}

\begin{proof}
For every $i\in [n]$ let $g'_i\colon A^{nm} \to A$ be the operation defined by $g'_i(\bfx)=g_i(\bfx_i)$ 
where $\bfx_i=(x_{(i-1)m+1}, \ldots, x_{im})$. We prove that $f(g'_1, \ldots, g'_n)$ is $\Pi$-decomposable. 
For $\bfx\in A^{nm}$,  set 
\begin{align*}
\bfc_\bfx^0& = f(0,g'_2(\bfx), \ldots, g'_n(\bfx)), & \qquad
\bfc_\bfx^1&= f(1,g'_2(\bfx), \ldots, g'_n(\bfx)),\\
\bfa_\bfx^0&=g'_1(0,x_2, \ldots, x_{nm}), & \qquad
\bfa_\bfx^1&=g'_1(1,x_2, \ldots, x_{nm}).
\end{align*}
 We obtain by $\Pi$-decomposability of $f$ that \[f(g'_1(\bfx), \ldots, g'_n(\bfx))= \Pi(g'_1(\bfx),\bfc_\bfx^1, \bfc_\bfx^0).\]
By iterating the pivotal decomposition expression (to each argument), we get the following equalities
\begin{align*}
 \Pi(g'_1(\bfx),\bfc_\bfx^1, \bfc_\bfx^0) 
 & =  \Pi(\Pi(x_1, \bfa_\bfx^1,\bfa_\bfx^0),\bfc_\bfx^1, \bfc_\bfx^0),\notag\\
& =  \Pi(x_1, \Pi(\bfa_\bfx^1, \bfc_\bfx^1, \bfc_\bfx^0),\Pi(\bfa_\bfx^0, \bfc_\bfx^1, \bfc_\bfx^0) ),\notag\\
&=\Pi(x_1, f(g'_1, \ldots, g'_n)(\bfx_1^1), f(g'_1, \ldots, g'_n)(\bfx_1^0))\notag,
\end{align*}
where the first equality is obtained by $\Pi$-decomposability of $g'_1$, the second one by equation \eqref{eqn:ex04} and the last one by $\Pi$-decomposability of $f$.
Thus, we have proved that condition \eqref{eqn:dec} holds for $f(g'_1, \ldots, g'_n)$ and $i=1$. We can proceed in a similar way  to obtain
\begin{equation}\label{eq:bnh}
f(g'_1(\bfx), \ldots, g'_n(\bfx)) = \Pi(x_\ell, f(g'_1, \ldots, g'_n)(\bfx_\ell^1),\\ f(g'_1, \ldots, g'_n)(\bfx_\ell^0)), 
\end{equation}
for every $\ell \in [nm]$.
The decomposability of $\Pi(g_1, \ldots, g_n)$ follows from \eqref{eq:bnh} by identifying all arguments in $\{x_i, x_{m+i}, \ldots, x_{(n-1)m+i}\}$ for every $i\in [m]$.
\end{proof}

Similarly, if the pivotal operation satisfies equation \eqref{eqn:ex01}, then  $\Lambda_\Pi$ must contain all projections.

%{\color{red}
\begin{lemma}\label{lem:sta03}
Let  $\Pi$ be a pivotal operation. The following conditions are equivalent.
\begin{enumerate}[(i)]
\item\label{it:lkj01} $\Pi$ satisfies equation \eqref{eqn:ex01}.
\item\label{it:lkj02} $\Lambda_\Pi$ contains all  projections on $A$.
\item\label{it:lkj03} $\Lambda_\Pi$ contains the unary projection $p_1^1$.
\end{enumerate}
\end{lemma}
%}

\begin{proof}
\eqref{it:lkj01} $\implies$ \eqref{it:lkj02}:
Let $n\geq 1$ and $k \in [n]$. For every $i\in [n]$ such that $i\neq k$ and for every $\bfx \in A^n$,
\[
\Pi(x_i, p^n_k(\bfx_i^1), p^n_k(\bfx_i^0))=\Pi(x_i, x_k, x_k)=x_k
\]
where the last equality is obtained by \eqref{eqn:pri}. If $i=k$, then
\[
\Pi(x_i, p^n_k(\bfx_i^1), p^n_k(\bfx_i^0))=\Pi(x_k,1,0)=x_k
\]
where the last equality is obtained by \eqref{eqn:ex01}. We conclude that $p_k^n\in \Lambda_\Pi$.

\eqref{it:lkj02} $\implies$ \eqref{it:lkj03}: Trivial. 

\eqref{it:lkj03} $\implies$ \eqref{it:lkj01}:  If $\Lambda_\Pi$ contains the unary projection $p_1^1$, then for every $x\in A$ we have 
\[
x =p_1^1(x)=\Pi(x,1,0).
\]
Thus $\Pi$ satisfies equation \eqref{eqn:ex01}, and the proof of the lemma is now complete.
\end{proof}

By combining Lemmas \ref{lem:sta01}, \ref{lem:sta02}, and \ref{lem:sta03}, we obtain the following result.

\begin{proposition}\label{prop:clo}
Suppose that $\Pi$ is a pivotal operation that satisfies  equation  \eqref{eqn:ex04}. Then $\Lambda_\Pi$ is a clone if and only if 
$\Pi$ satisfies  equation  
\eqref{eqn:ex01}. In the latter case, $\Lambda_\Pi$ is a clone that contains all constant operations.
\end{proposition}

We illustrate the previous results by  analyzing  the particular  case of Boolean functions.
\begin{example}\label{ex:bool}
Let $\Pi$ be a Boolean pivotal operation such that $\Lambda_\Pi$ is a clone. According to Proposition \ref{lem:sta03}, the operation $\Pi$ satisfies equation  \eqref{eqn:ex01}.
Hence, the unary sections $\Pi(x,0,0)$ and $\Pi(x,1,1)$ are determined by \eqref{eqn:pri} while the value of the section $\Pi(x,1,0)$ is determined by  \eqref{eqn:ex01}:
\begin{equation}\label{eqn:booltriv}
\Pi(x,0,0)=0, \qquad \Pi(x,1,1)=1, \quad \Pi(x,1,0)=x.
\end{equation}
Moreover, it is not difficult to check that the four possibilities for the unary section $\Pi(x,0,1)$, namely,
\begin{align*}
\Pi_0(x,0,1)=x, & \qquad
\Pi_1(x,0,1)=\overline{x},\\
\Pi_2(x,0,1)=0, & \qquad
\Pi_3(x,0,1)=1,
\end{align*}
 give rise to  operations $\Pi_0, \ldots, \Pi_3$ that satisfy equation \eqref{eqn:ex04}. 
Simple computations then  show that we must have 
\begin{align*}
\Pi_0(x,y,z)&=(x\wedge y)\vee (x\wedge z) \vee (y \wedge z), & \qquad \Pi_1(x,y,z)&= (x \wedge y) \vee (\overline{x} \wedge z),\\
%\Pi_1(x,y,z)&= (x \wedge y) \vee (z \wedge (\overline{x} \vee y)),\\
\Pi_2(x,y,z)&=y \wedge (x \vee z), & \qquad
\Pi_3(x,y,z)&=z \vee (x \wedge y).
\end{align*}
Hence, the clones $\Lambda_{\Pi_0}, \ldots, \Lambda_{\Pi_3}$ are as follows:
\begin{enumerate}[(a)]
 \item $\Lambda_{\Pi_0}$ is the clone $M$ of all monotone Boolean functions, since $\Pi_0=\med$;
 \item $\Lambda_{\Pi_1}$ is the clone $\mathcal{O}_{\{0,1\}}$ of all Boolean functions,  since $\Pi_1$ is the pivotal operation used in Shannon decomposition;
  \item $\Lambda_{\Pi_2}$ is the clone $M$ of all monotone Boolean functions, since $\Pi_2(x,y,0)=y\wedge x$ and $\Pi_2(x,1,z)=x\vee z$, and every composition 
  of $\Pi_2$ with projections or constants is monotone;
  \item $\Lambda_{\Pi_3}$ is the clone $M$ of all monotone Boolean functions (by a similar argument to that used for  $\Lambda_{\Pi_2}$).
\end{enumerate}

The situation can be summarized by the following result.
\end{example}
\begin{proposition}
If $\mathcal{C}$ is Boolean clone, then there is a Boolean pivotal operation $\Pi$ such that $\mathcal{C}=\Lambda_\Pi$ if and only if $\mathcal{C}$ is the clone of all monotone Boolean functions or the clone of all Boolean functions.
\end{proposition}

\subsection{The case of a pivotally decomposable $\Pi$}\label{sub:nec}
In the section, we derive results about clones of $\Pi$-decomposable operations under the additional assumption that the operation $\Pi$ itself is $\Pi$-decomposable, i.e., that $\Pi \in \Lambda_\Pi$. 

The next result states that under this assumption, the pivotal operation 
together with constant maps suffice to construct expressions representing each member of $\Lambda_\Pi$.

\begin{proposition}\label{prop:piv}
Let $\Pi$  be a pivotal operation such that $\Lambda_\Pi$ is a clone that contains $\Pi$. 
Then $\Lambda_\Pi$ is the clone generated by $\Pi$ and the constant maps. In particular, $\Lambda_\Pi=N_\Pi$.
\end{proposition}

\begin{proof}
Let $C$ be the clone generated by $\Pi$ and the constant operations. We have to prove  that $\Lambda_\Pi=C$. 
The right to left inclusion is trivial since $\Lambda_\Pi$ is a clone and contains each of the mentioned generators of $C$ by assumption and Lemma \ref{lem:sta01}. 
We derive the converse inclusion and the last part of the statement from the following sequence of inclusions,
\[
\Lambda_\Pi \subseteq N_\Pi \subseteq C \subseteq \Lambda_\Pi,
\]
where the first inclusion is obtained by Proposition \ref{prop:norm}, the second inclusion follows from the definitions of $N_\Pi$ and $C$, and the third inclusion
is a consequence of the first part of this proof.
%
%Conversely, if $f\in \Lambda_\Pi$, then by Proposition \ref{prop:norm}, $f$ is equal to a $\Pi$-normal form, which is a composition of $\Pi$ and $0$-ary operations. 
\end{proof}

In the presence of a $\Pi$-decomposable operation $\Pi$, equation \eqref{eqn:ex01}  has interesting consequences on the equational theory of the the algebra $\struc{A, \Pi,0,1}$, where $\Pi$ is a pivotal operation. 

\begin{lemma}
If $\Pi$ is a pivotal operation on $A$ that satisfies \eqref{eqn:ex01}, then it satisfies the following equations:
\begin{align}\Pi(0,1,z) & =z,\\
\Pi(1,1,z)& =1,\label{eqn:poi}\\
\Pi(0,y,0)& =0,\label{eqn:poibis}\\
\Pi(1,y,0)& =y.
\end{align}
\end{lemma}
\begin{proof}
The proof follows from straightforward applications of equations \eqref{eqn:dec} and \eqref{eqn:ex01}. For instance, we obtain successively 
\begin{equation*}
\Pi(0,1,z) =  \Pi(z, \Pi(0,1,1), \Pi(0,1,0)) =  \Pi(z,1,0) = z,
\end{equation*}
where the first equality is obtained by \eqref{eqn:dec} and the two last ones by \eqref{eqn:ex01}.
\end{proof}

 According to Proposition  \ref{prop:clo}, if $\Pi$ is a pivotal operation that satisfies equations \eqref{eqn:ex01} and \eqref{eqn:ex04}, then $\Lambda_\Pi$ is a clone. In the next theorem, we prove that the converse statement also holds, under the assumption that $\Pi\in \Lambda_\Pi$ and that
\begin{align}
\Pi(\Pi(1,0,1),0,1) & =\Pi(1, \Pi(0,0,1), \Pi(1,0,1)),\label{eqn:fine01}\\
\Pi(\Pi(0,0,1),0,1) & = \Pi(0, \Pi(0,0,1), \Pi(1,0,1))\label{eqn:fine02}.
\end{align}

\begin{theorem}\label{thm:clo}
Let $\Pi$ be  a $\Pi$-decomposable pivotal operation that satisfies \eqref{eqn:fine01} and \eqref{eqn:fine02}. The following conditions are equivalent: 
\begin{enumerate}[(i)]
 \item\label{it:poi01} $\Lambda_\Pi$ is a clone,
 \item\label{it:poi02} $\Pi$ satisfies  equations \eqref{eqn:ex01} and  \eqref{eqn:ex04}.
\end{enumerate}
In this case, $\Lambda_\Pi$ is the clone generated by $\Pi$ and the constant maps.
\end{theorem}
\begin{proof}
Proposition \ref{prop:clo} states that (\ref{it:poi02})~$\implies$~(\ref{it:poi01}). Conversely, assume that $\Pi$ is a pivotal operation such that $\Lambda_\Pi$ is a clone that contains $\Pi$. By Lemma \ref{lem:sta03}, it follows that $\Pi$ satisfies equation \eqref{eqn:ex01}.

 We prove that \eqref{eqn:ex04} also holds. In what follows, we use without further warning the fact that $\Lambda_\Pi$ is a clone that contains $\Pi$ and every constant operation to apply  \eqref{eqn:dec} to operations which are compositions of $\Pi$ and constant ones.

Hence, if $L(x,y,z)$ and $R(x,y,z)$ denote the operations given by the left-hand side and right-hand side of equation \eqref{eqn:ex04}, respectively, we have
\begin{align*}
L(x,y,z) & = \Pi\big(x, \Pi\big(\Pi(1,y,z),t,u\big),\Pi\big(\Pi(0,y,z),t,u\big)\big),\\
R(x,y,z) & = \Pi\big(x, \Pi\big(1,\Pi(y,t,u), \Pi(z,t,u)\big), \Pi\big(0,\Pi(y,t,u), \Pi(z,t,u)\big)\big).
\end{align*}
To prove that $L(x,y,z)=R(x,y,z)$ it suffices to prove that the two following equations hold:
\begin{align}
\Pi\big(\Pi(1,y,z),t,u\big) & = \Pi\big(1,\Pi(y,t,u), \Pi(z,t,u)\big),\label{eqn:lkj01}\\
\Pi\big(\Pi(0,y,z),t,u\big) & = \Pi\big(0,\Pi(y,t,u), \Pi(z,t,u)\big).\label{eqn:lkj02}
\end{align}
We prove that \eqref{eqn:lkj01} holds (with the help of \eqref{eqn:fine01}). Equation \eqref{eqn:lkj02} can be obtained in a similar way (with the help of \eqref{eqn:fine02}).

By decomposing with respect to $y$, we obtain that the right-hand side of \eqref{eqn:lkj01} is equal to 
\begin{equation*}
 \Pi\big(y, \Pi\big(\Pi(1,1,z),t,u\big), \Pi\big(\Pi(1,0,z),t,u\big)\big),
\end{equation*}
while the left-hand side of \eqref{eqn:lkj01} is equal to
 \begin{equation*} \Pi\big(y,  \Pi\big(1,\Pi(1,t,u), 
 \Pi(z,t,u)\big),\\ \Pi\big(1,\Pi(0,t,u), \Pi(z,t,u)\big)\big).\label{eqn:lkj04} 
\end{equation*} 
Hence, to prove that equation \eqref{eqn:lkj01} holds, we first  observe that 
\begin{align*}
 \Pi\big(\Pi(1,1,z),t,u\big) & = \Pi(1,t,u) \qquad \qquad \text{by \eqref{eqn:poi}},\\
& = \Pi\big(z,\Pi(1,t,u),\Pi(1,t,u)\big)  \qquad \qquad \text{by \eqref{eqn:pri}},\\
& = \Pi\big(z, \Pi\big(1,\Pi(1,t,u),\Pi(1,t,u)\big),\Pi\big(1,\Pi(0,t,u),\Pi(1,t,u)\big)\big),\ \text{by \eqref{eqn:dec}, \eqref{eqn:pri},}\\
& = \Pi\big(1, \Pi(1,t,u), \Pi(z,t,u)\big)  \qquad \qquad \text{by \eqref{eqn:dec}}.
\end{align*}
It remains to prove that
\begin{equation}\label{eqn:hgt}
\Pi\big(\Pi(1,0,z),t,u\big) = \Pi\big(1, \Pi(0,t,u), \Pi(z,t,u)\big).
\end{equation}
By decomposing with regard to $z$ we obtain
\begin{align*}
\Pi\big(\Pi(1,0,z),t,u\big) & = \Pi\big(z, \Pi\big(\Pi(1,0,1),t,u\big), \Pi(0,t,u)\big),\\
\Pi\big(1, \Pi(0,t,u), \Pi(z,t,u)\big) & =  \Pi\big(z, \Pi\big(1, \Pi(0,t,u), \Pi(1,t,u)\big),  \Pi(0,t,u)\big),
\end{align*}
and it suffices to prove that
\begin{equation}\label{eqn:lkh}
\Pi\big(\Pi(1,0,1),t,u\big) = \Pi\big(1, \Pi(0,t,u), \Pi(1,t,u)\big).
\end{equation}
Observe that by decomposing with respect to $u$,
\begin{align*}
\Pi\big(\Pi(1,0,1),t,u\big)  &= \Pi\big(u, \Pi\big(\Pi(1,0,1),t,1\big),\Pi\big(\Pi(1,0,1),t,0\big)\big),\\
 \Pi\big(1, \Pi(0,t,u), \Pi(1,t,u)\big) & = \Pi\big(u,\Pi\big(1, \Pi(0,t,1), \Pi(1,t,1)\big), \Pi\big(1, 0, \Pi(1,t,0)\big)\big),
\end{align*}
where we have applied \eqref{eqn:poibis} to obtain the second  identity. Hence, to prove \eqref{eqn:lkh} it suffices to prove that
\begin{align}
\Pi\big(\Pi(1,0,1),t,1\big) & = \Pi\big(1, \Pi(0,t,1), \Pi(1,t,1)\big),\label{hgf01}\\
\Pi\big(\Pi(1,0,1),t,0\big) & =\Pi\big(1, 0, \Pi(1,t,0)\big).\label{hgf02}
\end{align}
By \eqref{eqn:dec} we obtain
\begin{equation*}
\Pi\big(\Pi(1,0,1),t,0\big)= \Pi\big(t, \Pi(1,0,1), 0\big)=
\Pi\big(1, 0, \Pi(1,t,0)\big), 
\end{equation*}
which proves \eqref{hgf02}. Next, we observe that
\begin{align*}
\Pi\big(\Pi(1,0,1),t,1\big) & = \Pi\big(t, 1, \Pi\big(\Pi(1,0,1),0,1\big)\big)\\
 \Pi\big(1, \Pi(0,t,1), \Pi(1,t,1)\big) & = \Pi\big(t, 1, \Pi\big(1, \Pi(0,0,1), \Pi(1,0,1)\big)\big).
\end{align*}
We conclude that \eqref{hgf01} is satisfied by applying \eqref{eqn:fine02}, which holds by assumption. 
\end{proof}

Since \eqref{eqn:fine01} and \eqref{eqn:fine02} are instances of \eqref{eqn:ex04}, Theorem \ref{thm:clo} can be restated as follows. 
\begin{corollary}\label{cor:main}
Let  $\Pi$ be a $\Pi$-decomposable pivotal operation. 
The following conditions are equivalent: 
\begin{enumerate}[(i)]  
 \item\label{it:p01} $\Lambda_\Pi$ is a clone and $\Pi$ satisfies \eqref{eqn:fine01} and \eqref{eqn:fine02},
 \item\label{it:p02} $\Pi$ satisfies \eqref{eqn:ex01} and \eqref{eqn:ex04}.
\end{enumerate}
\end{corollary}

By noting that equations  \eqref{eqn:fine01} and \eqref{eqn:fine02} are satisfied by a symmetric pivotal operation that satisfies  \eqref{eqn:ex01}, we obtain the following corollary.

\begin{corollary}
Let  $\Pi$ be a symmetric $\Pi$-decomposable pivotal operation that satisfies \eqref{eqn:ex01}. 
The following conditions are equivalent: 
\begin{enumerate}[(i)]
 \item\label{it:pl01} $\Lambda_\Pi$ is a clone,
 \item\label{it:pl02} $\Pi$ satisfies \eqref{eqn:ex04}.
\end{enumerate}
\end{corollary}

\subsection{Further issues and counter-examples}\label{sub:nec2}

In view of Theorem \ref{thm:clo} and Corollary \ref{cor:main}, a natural question arises: can the $\Pi$-decomposability of $\Pi$ be deduced from equations \eqref{eqn:ex01} and  \eqref{eqn:ex04}? 
Example \ref{ex:bool} shows that the answer is positive if $\Pi$ is a Boolean pivotal operation. Now, we prove that it is not true in general. We set $A_\Delta=\{(x,y)\in A^2\mid x\neq y \ \&\ (x,y)\neq (1,0)\}$.

\begin{proposition}\label{prop:ex}
Let $\Pi$ be a pivotal operation that satisfies  \eqref{eqn:ex01}. If there exits a function $f\colon A_\Delta \to A$ such that
$\Pi(x,y,z)=f(y,z)$ for every $(y,z)\in A_\Delta$, then  $\Pi$ also satisfies \eqref{eqn:ex04}.
\end{proposition}
\begin{proof}
Note first that  $\Pi$ is well defined by the conditions in the statement. Moreover, equations \eqref{eqn:pri} and \eqref{eqn:ex01} ensure that 
\eqref{eqn:ex04} is satisfied when $t=u$ or $(t,u)=(1,0)$, respectively. Now, if $(t,u)\in A_\Delta$, then 
\begin{gather*}
%\Pi(x,1,0)=x \label{eqn:ex01},\\
\Pi(\Pi(x,y,z),t,u)=f(t,u)=\Pi(x,f(t,u), f(t,u))=\Pi(x,\Pi(y,t,u), \Pi(z,t,u)).
\end{gather*}
This shows that \eqref{eqn:ex04} does indeed hold for such a $\Pi$. 
\end{proof}
\begin{example}
Assume that $A$ has at least three elements $0,1,2$. Let $f\colon A_\Delta \to A$ be any mapping that satisfies $f(2,0)=f(2,1)\neq f(2,2)$. Then the pivotal operation $\Pi$ defined as in Proposition \ref{prop:ex} is not $\Pi$-decomposable since $\Pi(1,2,2)=f(2,2)$ while $\Pi(2, \Pi(1,2,1), \Pi(1,2,0))=f(2,0)$.
\end{example}

Theorem \ref{thm:clo} gives a characterization of pivotal operations $\Pi$ such that $\Lambda_\Pi$ is a clone, under the assumption 
that $\Pi\in \Lambda_\Pi$. We now give an example of a pivotal operation $\Pi$ such that $\Lambda_\Pi$ is a clone that does not contain $\Pi$.  
Note however that Example \ref{ex:bool} shows that such a $\Pi$ does not exist in the case of Boolean functions.

\begin{example}\label{ex:clonenopi}
Let $A=\{0,a, 1\}$ and $N\colon A\to A$ be the map defined by $N(0)=1$, $N(a)=a$, and $N(1)=0$. Define $\Pi\colon A^3 \to A$ as the map that satisfies \eqref{eqn:ex01}, \eqref{eqn:pri} and
\begin{align}
\Pi(x,0,1)&=N(x)\label{eqn:oo01}\\
\Pi(x,1,a)&=1\label{eqn:oo02}\\
\Pi(x,0,a)&=1\label{eqn:oo03}\\
\Pi(x,a,1)&=0\label{eqn:oo04}\\
\Pi(x,a,0)&=0.\label{eqn:oo05} 
\end{align}
First, observe that $\Pi\not\in\Lambda_\Pi$. Indeed, for any $x\in A$ we have on the one hand $\Pi(x,a,a)=a$ while $\Pi(a, \Pi(a,\Pi(x,1,a), \Pi(x,0,a))=\Pi(a,1,1)=1$. According to Proposition \ref{prop:clo}, it suffices to prove that $\Pi$ satisfies equation \eqref{eqn:ex04}. If  $t=u$ or $(t,u)\in \{(a,0), (a,1), (1,a),(0,a)\}$ then $\Pi(x,t,u)$ is constant and \eqref{eqn:ex04} holds trivially. If $(t,u)=(1,0)$ then $\Pi(x,t,u)$ is the first projection and  \eqref{eqn:ex04} holds as well. It remains to consider the case $(t,u)=(0,1)$. We have to prove
\begin{equation}\label{eq:iiy}
\Pi(x, N(y), N(z))~=~N(\Pi(x,y,z)).
\end{equation}
If $y=z$ then or $(y,z)\in \{(a,0), (a,1), (1,a),(0,a)\}$ then\eqref{eq:iiy} clearly holds by \eqref{eqn:oo02}~-~\eqref{eqn:oo05}. If $(y,z)\in\{(0,1),(1,0)\}$ then  \eqref{eq:iiy} holds by \eqref{eqn:oo01}.
%Next, we prove that $\Lambda_\Pi$ is the clone generated by the function $N\colon A\to A$ defined by $N(0)=1$, $N(a)=a$, and $N(1)=0$, \emph{i.e.}, $\Lambda_{\Pi}$ is exactly made of the constant functions, the projections and the function $N$.
%
%First, we prove that if $g\in \Lambda_\Pi$ is essentially unary, then $g$ is the identity of the function $N$. This result follows from the following analysis. Let $g\in \Lambda_\Pi$ be a non constant function different from the identity and the function $N$.
%\begin{enumerate}
%\item\label{it:dfr01} If $g(1)=a$ and $g(0)=a$, then $g$ is the constant map $a$, which violates our assumption. If $g(1)=a$ and $g(0)\neq a$,  then $\Pi(1, g(1),g(0))=\Pi(1,a,g(0))=0$ which is a contradiction since $g\in \Lambda_\Pi$. Similarly, if $g(0)=a$ and $g(1)\neq a$ then $\Pi(0, g(1), g(0))=1$ and we obtain the same contradiction.
%\item From (\ref{it:dfr01}), we may assume that $g(0)\neq a$ and $g(1)\neq a$. If $g$ is one-to-one, then either $g$ is the identity function or is equal to $N$, a contradiction. If $g$ is not one-to-one, since $g$ is non-constant, we have  $g(a)=a$ and one the following cases:
%\begin{align}
% g(0)& =g(a)\in\{0,1\}, \label{it:dff01}\\
%g(1)&=g(a)\in\{0,1\},\label{it:dff02}
%% g(0)&=g(1) % This means $g$ is constant.
%\end{align}
%(the case $g(0)=g(1)$ would lead to a constant function $g$). In any of these cases, we have $\Pi(a,g(0),g(1))\in\{0,1\}$ which is in contradiction with the $\Pi$-decomposability of $g$ since $g(a)=a$.
%\end{enumerate}
\end{example}

\section{Conclusions and Further Research}\label{sec:final}
 In this paper, we studied pivotal decompositions of operations from a clone theory perspective, and presented a 
 characterization of classes of 
 $\Pi$-decomposable operations that are clones in the case when the pivotal operation $\Pi$ is itself $\Pi$-decomposable and satisfies \eqref{eqn:fine01} and \eqref{eqn:fine02}. 
 However in Example \ref{ex:clonenopi} we showed that there exists a clone of $\Pi$-decomposable operations that does not contain $\Pi$, i.e.,
  $\Pi$ is not $\Pi$-decomposable. This leaves open a complete description of classes of  pivotally decomposable operations that are clones. 
  Moreover, once such a description is  obtained, a structural analysis of the set of all pivotally decomposable clones is to be expected.

Another topic that will deserve our attention is motivated by Theorem \ref{thm:symsy} that states that if a pivotal operation $\Pi$ 
is a $\Pi$-decomposable and satisfies $\Pi(x,1,0)=x$, $\Pi(0,0,1)=0$, and 
$\Pi(1,0,1)=1$, then  $\Pi$ is symmetric and hence is a majority operation. 
Furthermore, if $\Pi$ satisfies \eqref{eqn:ex04},  then $\Pi$ is a median operation (see \cite{Bandelt1983} and the bibliography therein). 
These observations establish noteworthy connections between pivotally decomposable classes and median algebras, and should deserve 
a deeper study in future research. 

As a third line of research that emerges from this paper deals with normal form representations of operations arising from pivotal decomposition schemes.
Proposition \ref{prop:norm} provides  normal form representations for the elements of a 
pivotally decomposable class. Here, determining canonical expressions for these representations based on the pivotal operation, 
as well as studying the complexity of such representations (e.g., with respect to classical normal form representations)
constitute an interesting topic of research which is under current investigation. 
We envision a similar study to that of  \cite{CFL} where, in particular, it was shown that normal form representations of Boolean functions 
that use the ternary median as the only logical connective, produce asymptotically 
shorter representations than the classical DNF, CNF and polynomial representations.

\section*{Acknowledgment}

This work was supported  by the internal research project  F1R-MTH-PUL-15MRO3 of the University of Luxembourg.

% that's all folks
\end{document}